\documentclass[10pt,a4paper]{article} 
\usepackage{moreverb}

\usepackage{amsfonts}
\usepackage{amsmath}
\usepackage{amssymb}
\usepackage{enumerate}
\usepackage{amsthm} 
\usepackage[final]{graphicx} 
\usepackage{ifthen} 
\usepackage{graphpap} 
\usepackage{longtable}
\usepackage{time} 
\usepackage{xspace}

\newcommand{\RootPath}{.}

\newcommand{\ExternalFiguresPath}{\RootPath/figures}

\newcommand{\eop}{\hspace*{\fill}~$\square$} 

\theoremstyle{plain}
\numberwithin{equation}{section}
\newtheorem{theorem}{Theorem}[section]
\newtheorem{proposition}[theorem]{Proposition}
\newtheorem{lemma}[theorem]{Lemma}

\newtheorem{Conjecture}[theorem]{Conjecture}

\newtheorem*{theorem*}{Theorem}
\newtheorem*{proposition*}{Proposition}
\newtheorem*{lemma*}{Lemma}
\newtheorem*{corollary*}{Corollary}
\newtheorem*{Conjecture*}{Conjecture}
\newtheorem*{unverified*}{Unverified Statement}
\newenvironment{conjecture*}{\begin{Conjecture*}}{\eop\end{Conjecture*}}

\theoremstyle{definition}
\newtheorem{Definition}[theorem]{Definition}

\newtheorem{Remark}[theorem]{Remark}

\newtheorem{Remarks}[theorem]{Remarks} 
\newtheorem{Example}[theorem]{Example}

\newtheorem{Examples}[theorem]{Examples}

\newtheorem{Hypothesis}[theorem]{Hypothesis}

\newtheorem{Problem}[theorem]{Problem}

\newtheorem{Problems}[theorem]{Problems}

\newtheorem{Excercise}[theorem]{Excercise}

\newtheorem*{Definition*}{Definition}
\newenvironment{definition*}{\begin{Definition*}}{\eop\end{Definition*}}
\newtheorem*{Remark*}{Remark}
\newenvironment{remark*}{\begin{Remark*}}{\eop\end{Remark*}}
\newtheorem*{Remarks*}{Remarks} 
\newenvironment{remarks*}{\begin{Remarks*} {\ } \hspace{-.5cm} \begin{enumerate}}{\end{enumerate}\eop\end{Remarks*}} 
\newtheorem*{Example*}{Example}
\newenvironment{example*}{\begin{Example*}}{\eop\end{Example*}}
\newtheorem*{Examples*}{Examples}
\newenvironment{examples*}{\begin{Examples*} \begin{enumerate}}{\end{enumerate}\eop\end{Examples*}}
\newtheorem*{Hypothesis*}{Hypothesis}
\newenvironment{hypothesis*}{\begin{Hypothesis*}}{\eop\end{Hypothesis*}}
\newtheorem*{Problem*}{Problem}
\newenvironment{problem*}{\begin{Problem*}}{\eop\end{Problem*}}
\newtheorem*{Problems*}{Problems}
\newenvironment{problems*}{\begin{Problems*} {\ } \hspace{-.5cm} \begin{enumerate}}{\end{enumerate}\eop\end{Problems*}}
\newtheorem*{Exercise*}{Exercise}
\newenvironment{exercise*}{\begin{Exercise*}}{\eop\end{Exercise*}}

\newcommand{\mycaption}[1]{\centering{\vspace{\medskipamount}\refstepcounter{figure}\textbf{Figure~\thefigure.} {#1}}}

\newcommand{\natur}{\ensuremath{\mathbb{N}}}
\newcommand{\real}{\ensuremath{\mathbb{R}}}

\newcommand{\sprod}[2]{\left<#1,#2\right>}

\newcommand{\setcond}[2]{\left\{ #1 : #2 \right\}} 

\newenvironment{FigTab}[2]{
	\begin{figure}[htb]
	\setlength{\unitlength}{#2}
	\begin{center}
	\begin{tabular}{#1}
}{
    \end{tabular}
    \end{center}
    \end{figure}
}

\newcommand{\IncludeGraph}[2]{
	\includegraphics[#1]{\ExternalFiguresPath/{#2}}
}

\newcommand{\AverkovEmail}{gennadiy.averkov@googlemail.com}

\newcommand{\range}[2]{{#1},\ldots,{#2}}
\newcommand{\setrange}[2]{\{\range{#1}{#2}\}}

\newcommand{\aff}{\mathop{\mathrm{aff}}\nolimits}

\newcommand{\diam}{\mathop{\mathrm{diam}}\nolimits}

\newcommand{\card}{\mathop{\mathrm{card}}\nolimits}

\newcommand{\const}{\mathop{\mathrm{const}}\nolimits}

\newcommand{\E}{\mathop{\mathbb{R}}\nolimits} 

\newcommand{\polynomfont}{\mathcal} 
\newcommand{\polp}{p} 
\newcommand{\prepp}{\polynomfont{P}} 
\newcommand{\overtwocond}[2]{\scriptsize \begin{array}{c} {#1}, \\ {#2} \end{array}} 
\newcommand{\dotvar}{\,\cdot\,} 

\newcommand{\lnorm}[2]{\left|{#1}\right|_{#2}} 
\newcommand{\enorm}[1]{\left|{#1}\right|} 
\newcommand{\MxL}{\left[} 
\newcommand{\MxR}{\right]} 

\newcommand{\ThmSource}[1]{\emph{(#1).}}
\newcommand{\ThmTitle}[2][]{\ifthenelse{\equal{#1}{}}{\emph{(#2).}}{\emph{(#2; #1).}}}
\newcommand{\notion}[2][]{\emph{#2}\xspace} 
 
\newcommand{\transp}{\top}

\newcommand{\mfor}{\ \mbox{for} \ }
\newcommand{\mand}{\ \mbox{and} \ }
\newcommand{\OneVec}{\mathrm{1\hspace{-0.26em}l}}
\newcommand{\floor}[1]{\lfloor #1 \rfloor}
\newcommand{\sti}{\mathop{s}}

\newcommand{\calX}{\mathcal{X}}
\newcommand{\calF}{\mathcal{F}}

\newcommand{\eps}{\varepsilon}
\renewcommand{\card}{\#}

\newcommand{\vertx}{\mathop{\mathrm{vert}}}

\newcommand{\mycite}[2]{\ifthenelse{\equal{#2}{}}{\cite{#1}}{\cite[#2]{#1}}\xspace}
\newcommand{\GroetschelHenk}[1][]{\mycite{MR1976602}{#1}}
\newcommand{\BosseGroetschelHenk}[1][]{\mycite{MR2166533}{#1}}
\newcommand{\Bernig}[1][]{\mycite{Bernig98}{#1}}
\newcommand{\Ziegler}[1][]{\mycite{MR1311028}{#1}}

\newcommand{\RAGbook}[1][]{\mycite{MR1659509}{#1}}
\newcommand{\BroeckerNew}[1][]{\mycite{MR765338}{#1}}
\newcommand{\BroeckerOld}[1][]{\mycite{MR1137812}{#1}}
\newcommand{\Scheiderer}[1][]{\mycite{MR1005003}{#1}}

\newcommand{\SchnBk}[1][]{\mycite{MR94d:52007}{#1}} 
\newcommand{\IneqBk}[1][]{\mycite{MR944909}{#1}}

\begin{document}
\title{Representing Simple $d$-Dimensional Polytopes by\\ $d$ Polynomials\footnote{The results of the paper were supported by the Research Unit 468 ``Methods of Discrete Mathematics for the Synthesis and Control of Chemical Processes'' funded by the German Research Foundation.}}
\date{\small \today}
\author{\small Gennadiy Averkov and Martin Henk}
\maketitle

\begin{abstract}
A polynomial representation of a convex $d$-polytope $P$ is a finite set $\{p_1(x),\ldots,p_n(x)\}$ of polynomials over $\E^d$ such that $P=\setcond{x \in \E^d}{p_1(x) \ge 0 \ \mbox{for every} \  1 \le i \le n}.$ By $\sti(d,P)$ we denote the least  possible number of polynomials in a polynomial representation of $P.$  It is known that $d \le \sti(d,P) \le 2d-1.$ Moreover, it is conjectured that $\sti(d,P)=d$ for all convex $d$-polytopes $P.$  We confirm this conjecture for simple $d$-polytopes by providing an explicit construction of $d$ polynomials that represent a given simple $d$-polytope $P.$
\end{abstract}

\newtheoremstyle{itsemicolon}{}{}{\mdseries\rmfamily}{}{\itshape}{:}{ }{}
\theoremstyle{itsemicolon}
\newtheorem*{msc*}{2000 Mathematics Subject Classification} 

\begin{msc*}
 Primary 14P05, 52B11; Secondary 52A20
\end{msc*}

\newtheorem*{keywords*}{Key words and phrases}

\begin{keywords*}
Cayley cubic, elementary symmetric polynomial, H-representation, real algebraic geometry, semi-algebraic set, theorem of Scheiderer and Br\"{o}cker
\end{keywords*}

\section{Introduction} \label{sec intro}

The Euclidean space of dimension $d \ge 2$ is denoted by $\E^d.$ The origin, scalar product, and norm in $\E^d$ are denoted by $o,$ $\sprod{\dotvar}{\dotvar},$ and $\enorm{\dotvar},$ respectively. In analytic expressions points of $\E^d$ are treated as real column vectors of length $d$. The transposition is denoted by $(\dotvar)^\transp.$
	
Let $x$ be a vector variable in $\E^d.$ Given a finite set $\prepp$ of polynomials from $\real[x],$ the sets 
\begin{align*}
	S_0&:=\setcond{x \in \E^d}{\polp(x) > 0 \ \forall  \, \polp \in \prepp} &  &\mand &  S&:=\setcond{x \in \E^d}{\polp(x) \ge 0 \ \forall \, \polp \in \prepp}
\end{align*} are called \emph{basic open} and \emph{basic closed semi-algebraic set} represented by $\prepp,$ respectively. Let $\sti(d,S_0)$ and $\sti(d,S)$ stand for the least cardinality of a set of  polynomials representing $S_0$ and $S,$ respectively. It is known that
\begin{eqnarray}
	\max_{S_0} \sti(d,S_0) & = & d, \label{mSo:bound} \\
	\max_S \sti(d,S) & = & d(d+1)/2. \label{mS:bound}
\end{eqnarray}
This was shown by Br\"ocker and Scheiderer \BroeckerNew, \Scheiderer, \BroeckerOld, \RAGbook[\S6.5, \S10.4]; some extensions are given in \cite[Chapter~5]{MR1393194}, and a modified proof is presented in \cite{MR1053289} and \cite{MR1609085}.  The known proofs of \eqref{mSo:bound} and \eqref{mS:bound} are non-constructive. More precisely, explicit procedures for constructing the sets of polynomials representing a general $S_0$ (resp. $S$) and having cardinality at most $d$ (resp. $d(d+1)/2$) are not known, since the available proofs are based on some non-constructive existence theorems. 

A set $P$ in $\E^d$ is a \notion{convex polyhedron} if it is a non-empty intersection of a finite number of half-spaces. A convex polyhedron $P \subseteq \E^d$ is said to be a \notion{convex polytope} if it is bounded and a \notion{$d$-polytope} if it is bounded and of dimension $d.$ In this paper we study the quantity $\sti(d,P),$ where $P$ is a $d$-polytope. A $d$-polytope is said to be \notion{simple} if each of its vertices is contained in precisely $d$ facets. By $\vertx(P)$ we denote the set of all vertices of $P.$ We refer to \Ziegler for the background information on convex polytopes.  A set of polynomials representing a convex polyhedron $P$ in $\E^d$ is called a \notion{polynomial representation} of $P.$ Thus, polynomial representations are generalization of \notion{H-representations,} cf. \Ziegler[p.~28]. In \GroetschelHenk[Section~5] and \BosseGroetschelHenk[Section~4] it is mentioned  that one might be able to develop efficient solution techniques for some combinatorial optimizations problems by passing from H-representations to more general polynomial representations provided the degrees of the involved polynomials are not too high. 

Let us enumerate known constructive results on $\sti(d,P),$   see also the survey \cite{Henk06PolRep}.
 Improving a result of vom Hofe \cite{vomHofe}  Bernig \Bernig showed that $\sti(2,P)=2$ for every convex polygon $P$ in $\E^2,$ see Section~\ref{sec:ex} for more details for that case. 
  For an arbitrary dimension Gr\"otschel and Henk \GroetschelHenk constructed $O(d^d)$ polynomials representing a simple $d$-polytope and pointed out the lower bound $\sti(d,P)\ge d$ for all $d$-polytopes $P.$ The smallest known upper bounds for $\sti(d,P)$ were given in \BosseGroetschelHenk and \cite{BosseDiss}. More precisely, in \BosseGroetschelHenk it was shown that 
\begin{itemize}
	\item $\sti(d,P) \le 2d-2$ for pointed $d$-dimensional cones, 
	\item $\sti(d,P) \le 2d-1$ for $d$-polytopes,
	\item $\sti(d,P) \le 2d$ for $d$-polyhedra.
\end{itemize} 
 Each of the above three bounds has a constructive proof. In \BosseGroetschelHenk[Section~1] it was conjectured that $\sti(d,P)=d$ for every convex $d$-polytope $P$ in $\E^d.$ The aim of  this paper is to  confirm this conjecture for the class of simple $d$-polytopes, see Theorem~\ref{p-rep:simp} below. We recall that a $d$-polytope is \notion{simple} if each of its vertices is incident with precisely $d$ facets.  Our construction involves \notion[elementary symmetric polynomial]{elementary symmetric polynomials} defined by
\begin{equation} \label{sigma_k def eq}
	\sigma_l(y) :=\sigma_l(y_1,\ldots,y_m) :=\sum_{\overtwocond{J \subseteq \{1,\ldots,m\}}{\card J = l}} \prod_{j \in J} y_{j},
\end{equation}
where $y:=[y_1,\ldots,y_m]^\transp \in \E^m$ and $\card$ stands for the cardinality. We also put $\sigma_0(y):=1$ and $\sigma_l(y):=0$ for $l<0$ and $l>m.$

\begin{theorem} \label{p-rep:simp}  Let $P$ be a simple $d$-polytope in $\E^d.$ Then $\sti(d,P)=d.$ Furthermore,  assume that $P$  has $m$ facets and is given by affine inequalities $q_1(x) \ge 0,\ldots, q_m(x) \ge 0.$  
 Then
\begin{equation} \label{P:repr}
	P = \setcond{x \in \E^d}{p_i(x)\ge 0 \ \mbox{for} \ 0 \le i \le d-1}. 
\end{equation}
where
\begin{eqnarray*}
	p_{d-1}(x) & := & \sigma_m(q_1(x),\ldots,q_m(x)), \\
	& \cdots &  \\
	p_i(x) &:=& \sigma_{m-d+i+1}(q_1(x),\ldots,q_m(x)), \\
	& \cdots & \\
	p_1(x) &:=& \sigma_{m-d+2}(q_1(x),\ldots,q_m(x))
\end{eqnarray*}
and 
\begin{equation*}
	p_0(x) := 1 - \sum_{v \in \vertx(P)} y_{v} \left( \frac{1}{d} \sum_{\overtwocond{j=1,\ldots,m}{q_j(v)=0}} \bigl(1 - \lambda_j q_j(x)\bigr)^{2k}  \right)^{2k}
\end{equation*}
with appropriate $k \in \natur,$ $y_v > 0$ and $\lambda_j> 0.$ 
 \eop
\end{theorem}

We notice that for $p_0(x)$ from Theorem~\ref{p-rep:simp}, $p_i(x)$ vanishes on each $i$-face of $P$ for $i \in \{0,\ldots,d-1\}.$  As a direct consequence of  Theorem~\ref{p-rep:simp} we obtain that the polynomials $p_i(x), \ 0 \le i \le d-1,$ from Theorem~\ref{p-rep:simp} represent the interior of $P.$ Thus, there exists a constructive proof of \eqref{mSo:bound} for the special case when $S_0$ is the interior of a simple polytope.

As a consequence of  the \notion{Positivstellensatz} it can be derived that every  polynomial $p(x)$ which is non-negative on $P$ can be represented by  
	\begin{equation*}
		p(x)=\sum_{l} f_l(x) \sum_{j=1}^m q_j(x)^{l(j)},
	\end{equation*}
	where $l$ ranges over maps from $\{1,\ldots,m\}$ to $\natur \cup \{0\}$ and $f_l$ are non-negative polynomials on $\E^d$ (see \RAGbook[p.~106]). In our construction the polynomials are even of a more specific type, namely, such that $f_l(x)=\const$ for every $l.$ It turns out that it is reasonable to consider the polynomials of these form, see \GroetschelHenk[p.~487], \cite{MR929582}, and \cite{MR1854339}. In fact, such polynomials were also used in the previous papers.

The paper is organized as follows. In Section~\ref{sec:ex} we illustrate the statement of Theorem~\ref{p-rep:simp} by several examples. In Section~\ref{proof:qualit} we give the proof of Theorem~\ref{p-rep:simp}.  Estimates which allow to explicitely determine the possible choice of the parameter $k$ involved in the construction of $p_0(x)$ are given in Section~\ref{sec:proof:quant}. 

\section{Examples of polynomial representations} \label{sec:ex}

Let us illustrate the case $d=2.$ This case was completely settled by Bernig. Since convex polygons are simple polytopes, the case $d=2$ is also covered by Theorem~\ref{p-rep:simp}. The polynomial $p_0(x)$ describes a semi-algebraic set $\setcond{x \in \E^d}{p_0(x) \ge 0}$ which is sufficiently close to $P.$  In \Bernig it was proved that if $P$ is a convex $m$-gon given by affine inequalities $q_1(x) \ge 0,\ldots, q_m(x) \ge 0,$ then a strictly concave polynomial $p_0(x)$ vanishing on each vertex of $P$ can be constructed such that $p_0(x)$ together with the polynomial $p_1(x):=q_1(x) \cdot\ldots \cdot q_m(x)$ form a polynomial representation of $P$ (see also Fig.~\ref{bern:2d:fig}).

	\begin{FigTab}{c}{0.6mm}
	\begin{picture}(52,50)
	\put(0,2){\IncludeGraph{width=50\unitlength}{BernigsConstruction.eps}}
	\put(20,25){$P$}
	\end{picture}
	\\
	\parbox[t]{0.95\textwidth}{\mycaption{Bernig's construction; the region shaded by \includegraphics[height=0.3cm]{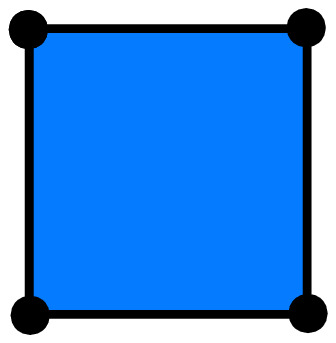} is $P ,$  \includegraphics[height=0.3cm]{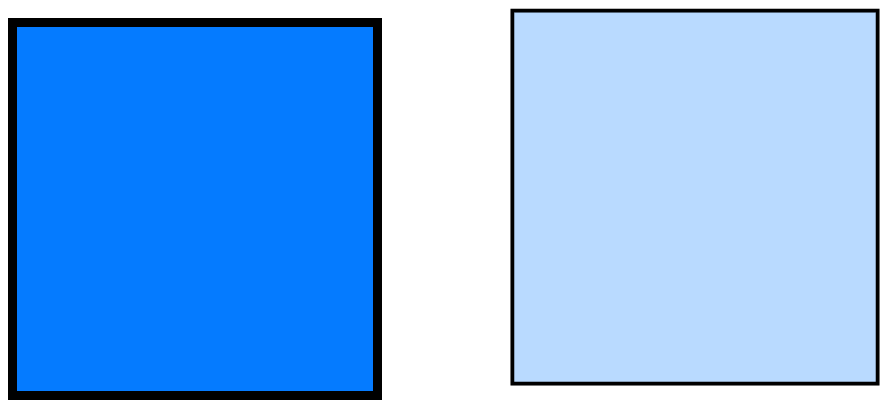}  is  $\setcond{x \in \E^2}{p_1(x)\ge 0}$ , \includegraphics[height=0.4cm]{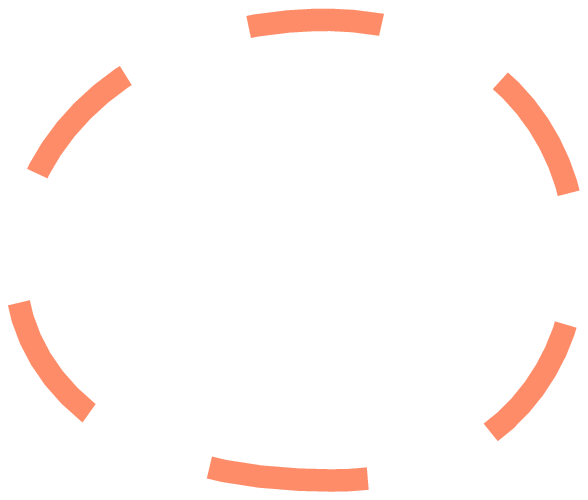} is the boundary of $\setcond{x \in \E^2}{p_0(x) \ge 0}$ \label{bern:2d:fig}}} 
	\end{FigTab}

We illustrate Theorem~\ref{p-rep:simp} for the case $d=3$ by some concrete choices of $P.$ For $J \subseteq \{0,\ldots,d-1\}$ with $J \ne \emptyset$  we use the notation $P_J:=\setcond{x \in \E^d}{p_j(x) \ge 0 \mfor j \in J}.$ By Theorem~\ref{p-rep:simp} one has $P=P_J$ for $J=\{0,\ldots,d-1\}.$ 

If $P$ is a regular tetrahedron with vertices 
\begin{align*}
	v_1 &:=[1,-1,1]^\transp, & v_2&:=[-1,1,1]^\transp, \\ \ v_3&:=[1,1,-1]^\transp, & v_4&:=[-1,-1,-1]^\transp,
\end{align*}
then we can choose 
\begin{align*}
	q_1(x) & :=    1+x_{{1}}-x_{{2}}+x_{{3}} , &
	q_2(x) & := 1-x_{{1}}+x_{{2}}+x_{{3}} ,\\
	q_3(x) & :=  1+x_{{1}}+x_{{2}}-x_{{3}}, &
	q_4(x) & :=  1-x_{{1}}-x_{{2}}-x_{{3}}.
\end{align*}
In this case 
\begin{align*}
	p_2(x) & =q_1(x) q_2(x) q_3(x) q_4(x) \\
		    &= 1-2\,{x_{{1}}}^
{2}-2\,{x_{{2}}}^{2}-2\,{x_{{3}}}^{2}-8\,x_{{1}}x_{{2}}x_{{3}}-2\,{x_{{1}}}^{2}{x_{{2}}}^{2}-2
\,{x_{{1}}}^{2}{x_{{3}}}^{2}-2\,{x_{{2}}}^{2}{x_{{3}}}^{2}+{x_{{1}}}^{4}+{x_{{2}}}^{4}+{x_{{3}}}^{4} \\
	p_1(x) &=q_1(x) q_2(x) q_3(x)+q_1(x) q_2(x) q_4(x) + q_1(x) q_3(x) q_4(x) + q_2(x) q_3(x) q_4(x) \\
		    &= 4 \, (1-{x_{{1}}}^{2}-{x_{{2}}}^{2}-{x_{{3}}}^{2}-2\,x_{{1}}x_{{2}}x_{{3}}),
\end{align*}
and thus the boundary of $P_1$ is the well-known \emph{Cayley cubic}. Fig.~\ref{TetrahedronDiag} depicts all possible $P_J$ in a diagram where an arrow is drawn from the image of $P_{J_1}$ to the image of $P_{J_2}$ whenever $J_1 \subseteq J_2.$ We wish to illustrate the properties of $p_1(x), p_2(x)$ from Theorem~\ref{p-rep:simp} rather than the properties of $p_0(x).$  Therefore, we choose $p_0(x)$  having a simpler form than in Theorem~\ref{p-rep:simp}, namely $p_0(x):=3-x_1^2-x_2^2-x_3^2$ so that $P_0$ is a ball of radius $\sqrt{3}.$ 

\begin{FigTab}{cc}{0.6mm}
\begin{picture}(132,160)
\put(35,55){\IncludeGraph{width=55\unitlength}{Tetrahedron.eps}}
	\put(-15,30){\IncludeGraph{width=55\unitlength}{TetrahedronP0.eps}}
\put(90,30){\IncludeGraph{width=55\unitlength}{TetrahedronP1.eps}}
\put(35,120){\IncludeGraph{width=55\unitlength}{TetrahedronP2.eps}}
\put(-15,90){\IncludeGraph{width=55\unitlength}{TetrahedronP02.eps}}
\put(90,90){\IncludeGraph{width=55\unitlength}{TetrahedronP12.eps}}
\put(35,0){\IncludeGraph{width=55\unitlength}{TetrahedronP01.eps}}
\put(5,10){\IncludeGraph{width=110\unitlength}{PDiagram.eps}}
\put(50,60){$P$}
\put(0,35){$P_0$}
\put(105,35){$P_1$}
\put(45,120){$P_2$}
\put(0,95){$P_{0,2}$}
\put(105,95){$P_{1,2}$}
\put(50,5){$P_{0,1}$}
\end{picture}
\\
\parbox[t]{0.9\textwidth}{\mycaption{\label{TetrahedronDiag}}}
\end{FigTab}

Now let $P$ be the cube given by $P:=\setcond{x \in \real^3}{|x_i| \le 1 \mfor 1 \le i \le 3}.$ Then we can take       
\begin{align*}
	q_1(x) &:= -x_1, & q_2(x)&:=-x_2, & q_3(x)&:=-x_3, \\
	q_4(x) &:= +x_1, & q_5(x)&:=+x_2, & q_6(x)&:=+x_3.	
\end{align*}
We have 
\begin{align*}
p_2(x) &= (1-{x_1}^2) (1 - {x_2}^3) (1-{x_3}^2), \\
p_1(x) &= 2(3-2\,{x_{{1}}}^{2}-2\,{x_{{2}}}^{2} -2\,{x_{
{3}}}^{2}+{x_{{1}}}^{2}{x_{{2}}}^{2}+{x_{{1}}}^{2}{x_{{3}}}^{2}+{x_{{2}}}^{2}{x_{{3}}}^{2}).
\end{align*}
We can choose $p_0(x)$ in the same way as for the previous example. The diagram depicting $P_J$ is given in Fig.~\ref{CubeDiag}. One can see that the boundary of $P_1$ is a surface sharing some properties with the Cayley cube, namely every vertex of $P$ is  the \notion{conic double point} of the mentioned surface. Thus, for a general simple $3$-polytope $P$ the boundary of $P_1$ can be viewed as a generalized Cayley surface assigned to $P.$ Singularities of algebraic surfaces are discussed in \cite[Section~5 of Chapter~I]{MR1336146}, \cite[Chapters~3,4]{MR1805816}, and \cite[Section~A.9]{MR1930604}. 

\begin{FigTab}{cc}{0.6mm}
\begin{picture}(132,160)
\put(35,55){\IncludeGraph{width=55\unitlength}{Cube.eps}}
\put(-15,30){\IncludeGraph{width=55\unitlength}{CubeP0.eps}}
\put(90,30){\IncludeGraph{width=55\unitlength}{CubeP1.eps}}
\put(35,120){\IncludeGraph{width=55\unitlength}{CubeP2.eps}}
\put(-15,90){\IncludeGraph{width=55\unitlength}{CubeP02.eps}}
\put(90,90){\IncludeGraph{width=55\unitlength}{CubeP12.eps}}
\put(35,0){\IncludeGraph{width=55\unitlength}{CubeP01.eps}}
\put(5,10){\IncludeGraph{width=110\unitlength}{PDiagram.eps}}
\put(45,55){$P$}
\put(-5,30){$P_0$}
\put(100,30){$P_1$}
\put(45,120){$P_2$}
\put(-5,90){$P_{0,2}$}
\put(100,90){$P_{1,2}$}
\put(45,0){$P_{0,1}$}
\end{picture}
\\
\parbox[t]{0.9\textwidth}{\mycaption{\label{CubeDiag}}}
\end{FigTab}

\section{The proof} \label{proof:qualit} 

\subsection{Preliminaries}


In what follows, $P$ is a $d$-polytope in $\E^d$ and $\calF_i$ denotes the class of all $i$-faces of $P.$ Given $F \in \calF_{d-1},$ $u_F$ stands for the outward unit normal of $P$  at the facet $F.$ By $\diam(P)$ we denote the diameter of $P,$ which is equal to the largest possible distance between two vertices of $P.$  With each $F \in \calF_{d-1}$ we associate affine functions 
\begin{equation*}
	q_F(x)  := \frac{h(P,u_F)-\sprod{u_F}{x}}{\diam(P)}, 
\end{equation*}
where 
$$
	h(P,u) := \max \setcond{ \sprod{x}{u} }{x \in P} , \qquad u \in \E^d,
$$
is the \notion{support function} of $P.$ We have $0 \le q_F(x) \le 1$ with $q_F(x)=0$ for all $x \in F.$ 
In what follows $m$ always denotes the number of facets in $P.$ 

In many cases we shall consider matrices and vectors indexed by the elements of $\calF_{d-1}$ and $\vertx(P)$ rather then by segments of natural numbers, which is possible if some linear order on each of these two classes is assumed to be fixed. For example, we introduce the affine mapping
\begin{eqnarray*}
	q(x) := [q_F(x)]_{F \in \calF_{d-1}} = [q_{F_1}(x),\ldots,q_{F_m}(x)]^\transp, 
\end{eqnarray*}
where $F_1,\ldots,F_m$ is a sequence of all facets of $P$ that determines an order on $\calF_{d-1}.$ For each $v \in \vertx(P)$ we also introduce the set
$$\calF_{d-1}^v:=\setcond{F \in \calF_{d-1}}{v \in F}$$
and the affine functions
\begin{eqnarray*}
	q_v(x) & := & [q_v(x)]_{v \in \calF_{d-1}^v}, \\
	\bar{q}_v(x) & := & [q_v(x)]_{v \in \calF_{d-1} \setminus \calF_{d-1}^v}.
\end{eqnarray*}


\subsection{Lemmas on $\eps_1, \ \eps_2, \ \eps_3$}

Given $\eps > 0$ consider the polytope
\begin{equation*}
	P_\eps := \setcond{x \in \E^d}{ q_F(x) \ge -\eps \mfor F \in \calF_{d-1}},
\end{equation*}
see also Fig.~\ref{PepsFig}.

	\begin{FigTab}{c}{0.6mm}
	\begin{picture}(42,36)
	\put(0,0){\IncludeGraph{width=40\unitlength}{Peps.eps}}
	\put(20,20){$P$}
	\put(9,0){$P_\eps$}
	\end{picture} 
	\\
	\parbox[t]{0.40\textwidth}{\mycaption{\label{PepsFig}}}
	\end{FigTab}

\newcommand{\SigmaQIneq}{\sigma_i(q(x)) \ge 0 \ \mbox{for} \ m-d+2 \le i \le m}

\begin{lemma} \label{07.06.19,09:55} Let $P$ be a simple $d$-polytope. Then there exists an $\eps_1>0$ such that  $\sigma_i(q(x))>0$ for $1 \le i \le m-d$ and $x \in P_{\eps_1}.$
\end{lemma}
\begin{proof}
	By \eqref{sigma_k def eq} we see that $\sigma_i(q(x))>0$ for $1 \le i \le m-d$ and $x \in P.$ In fact, in view of  \eqref{sigma_k def eq} the polynomial $\sigma_i(q(x))$ is given as a sum, where each summand is represented as a product of at most $m-d$ polynomials from the class $q_F(x), \ F \in \calF_{d-1}.$  But for $x \in P$ all values $q_F(x), \ F \in \calF_{d-1},$ are non-negative and at most $d$ values from $q_F(x), \ F \in \calF_{d-1},$ vanish. Consequenlty, at least one of the mentioned summands is strictly positive and hence $\sigma_i(q(x))>0.$ In view of the continuity of $\sigma_i(q(x))$ we obtain the assertion.
\end{proof}

	The non-negative orthant can be represented as the set where all elementary symmetric functions are non-negative. In \Bernig this statement was derived from the \notion{Descartes' rule of signs} (see    \RAGbook[Proposition~1.2.14] for the statement and a short proof of the Descartes' rule). Below we give an alternative direct proof.
\begin{proposition}  \ThmSource{Bernig, \Bernig[p.~38]} \label{orth repr}
	Let $d \ge 2.$ Then 
	\begin{equation*}
		\setcond{x \in \real^d}{x_1 \ge 0, \ldots, x_d \ge 0} = \setcond{x \in \real^d}{ \sigma_1(x) \ge 0, \ldots, \sigma_d(x) \ge 0}.
	\end{equation*}	
\end{proposition}
\begin{proof}
The inclusion ``$\subseteq$'' is trivial. Let us prove the reverse inclusion. Assume that $\sigma_i(x) \ge 0$ for $1 \le i \le d.$ Let $p(t):=(t+x_1)\cdots (t+x_d).$ By \notion{Vieta's formulas} 
\begin{equation*} 
 	p(t)= \sum_{i=0}^n \sigma_{n-i}(x) t^{i},
\end{equation*}
The polynomial $p(t)$ is not identically equal to zero. Since all its coefficients are non-negative, it cannot have positive real roots. Thus, all its roots $-x_j, 1 \le j \le d,$ are non-positive, and we are done.
\end{proof}

We observe that $\sigma_i(x) = O(|x|^i), \ 1 \le i \le d,$ since $\sigma_i(x) \le |x|^i \max \setcond{\sigma_i(u)}{u \in \E^d, \ |u|=1 }.$ Notice also that for $x \in \E^{n_1}, \ y \in \E^{n_2},$ and $z:=[x_1,\ldots,x_{n_1},y_1,\ldots,y_{n_2}]^\transp \in \E^{n_1+n_2}$ one has
\begin{equation} \label{sigma convol} 
	\sigma_i(z) = \sum_{j=-\infty}^{+\infty} \sigma_{i-j}(x) \sigma_{j}(y),
\end{equation}
where $1 \le i \le n_1+n_2.$ In \eqref{sigma convol} only the items with $0 \le i-j \le n_1$ and $0 \le j \le n_2$ (equivalently  $\max \{0,n_1-i\} \le j \le \min \{n_2,i\}$) can be non-zero.

Given $v \in \vertx(P)$ and $\eps > 0$ we introduce the sets 
\begin{eqnarray*}
	\Pi_{v,\eps} &:=& \setcond{x \in \E^d}{|q_v(x)|_\infty \le \eps}, \\
	C_v &:=&  \setcond{x \in \E^d}{-\sigma_1(q_v(x)) \ge \frac{2}{3} |q_v(x)| },
\end{eqnarray*}
see Figs.~\ref{PivepsFig}, \ref{CvFig}.

	\begin{FigTab}{cc}{0.8mm}
	\begin{picture}(30,38)
	\put(0,9){\IncludeGraph{width=25\unitlength}{Piveps.eps}}
	\put(14,20){$P$}
	\put(0,9){\small $\Pi_{v,\eps}$}
        \end{picture}
	&
	\begin{picture}(48,38)
	\put(0,2){\IncludeGraph{width=46\unitlength}{Cv.eps}}
	\put(6,8){\small $C_v$}
	\put(30,8){\small $2v-C_v$}
	\put(24,20){$P$}
	\put(14,14.5){\small $v$}
        \end{picture}
	\\
	\parbox[t]{0.40\textwidth}{\mycaption{\label{PivepsFig}}}
	&
	\parbox[t]{0.40\textwidth}{\mycaption{\label{CvFig}}}
	\end{FigTab}

It can be seen that  $\Pi_{v,\eps}$ is a small polytope enclosing $v.$ The set $C_v$ is a convex cone with apex at $v.$ This follows from the fact that the function $\frac{2}{3} |z| - \sigma_1(z), \ z \in \E^d,$ is sublinear (see \SchnBk[p.26]). Furthermore, $C_v \cap P=\{v\}$ and 
$$P \subseteq 2v-C_v = \setcond{x \in \E^d}{\sigma_1(q_v(x)) \ge \frac{2}{3} |q_v(x)|}.$$ 
Notice that $2v-C_v$ is the reflection of $C_v$ with respect to $v.$

\begin{lemma} \label{around:vx:lem}
	Let $P$ be a simple $d$-polytope. Then there exists an $\eps_2>0$ such that for every $v \in \vertx(P)$ 
	\begin{equation*} 
		\setcond{x \in \Pi_{v,\eps_2} }{ \sigma_i(q(x)) \ge 0 \mfor m-d+2 \le i \le m} \subseteq P \cup C_v.
	\end{equation*}
\end{lemma}
\begin{proof}
	Let $\eps_1$ be as in the statement of Lemma~\ref{07.06.19,09:55} and let us consider an arbitrary $v \in \vertx(P).$ We have
	\begin{align} 
	\sigma_{m-d+1}(q(x)) & \stackrel{\eqref{sigma convol}}{=} \sum_{i=1}^{+\infty} \sigma_i(q_v(x)) \sigma_{m-d+1-i}(\bar{q}_v(x)) = \sigma_1(q_v(x)) \sigma_{m-d}(\bar{q}_v(x)) + r_1(x) \label{07.03.02,15:51} \\
		\sigma_{m-d+2}(q(x)) & \stackrel{\eqref{sigma convol}}{=}  \sum_{i=2}^{+\infty} \sigma_i(q_v(x)) \sigma_{m-d+2-i}(\bar{q}_v(x)) = \sigma_2(q_v(x)) \sigma_{m-d}(\bar{q}_v(x)) + r_2(x) \nonumber \\
					& \stackrel{\phantom{\eqref{sigma convol}}}{=} \frac{1}{2} \sigma_{m-d}(\bar{q}_v(x)) \left( \sigma_1(q_v(x))^2-|q_v(x)|^2 \right) + r_2(x) \nonumber \\
						        & \stackrel{\phantom{\eqref{sigma convol}}}{=} g_1(x) \left( g_2(x) \sigma_1(q_v(x))^2 - |q_v(x)|^2 \right),  \label{07.03.02,13:30}
	\end{align}
	where the functions
	\begin{align*}
		r_1(x) &:= \sum_{i=2}^{+\infty} \sigma_i(q_v(x)) \sigma_{m-d+1-i}(\bar{q}_v(x)), & 
		g_1(x) &:=\frac{1}{2} \sigma_{m-d}(\bar{q}_v(x)) - \frac{r_2(x)}{|q_v(x)|^2}, \\
		r_2(x) & := \sum_{i=3}^{+\infty} \sigma_i(q_v(x)) \sigma_{m-d+2-i}(\bar{q}_v(x)),  &
		g_2(x) &:= \frac{\sigma_{m-d}(\bar{q}_v(x))}{\sigma_{m-d}(\bar{q}_v(x)) - 2 \frac{r_2(x)}{|q_v(x)|^2}}
	\end{align*}	
	are such that 
	\begin{align*}
		r_1(x) &=O(|q_v(x)|^2), & g_1(x) & \rightarrow \frac{1}{2} \sigma_{m-d}(\bar{q}_v(v)) > 0, \\
		r_2(x) &=O(|q_v(x)|^3), & g_2(x) & \rightarrow 1,
	\end{align*}
	as $x \rightarrow v.$ Consequently, we can choose an $\eps_v$ with $0 < \eps_v \le \eps_1$ such that for every $x \in \Pi_{v,\eps_v}$ 
	\begin{eqnarray}
		q_F(x) & > & 0 \qquad \mfor F \in \calF_{d-1} \setminus \calF_{d-1}^v, \label{eps:v assump} \\
		g_1(x)&>&0, \label{c2 ineq} \\
		g_2(x) &\le& \frac{9}{4}, \label{c1c2 ineq} \\
		|r_1(x)| & \le & \frac{1}{3} |q_v(x)| \sigma_{m-d}(\bar{q}_v(x)). \label{07.07.12,16:20}
	\end{eqnarray}
	From now on, let us assume that $x$ belongs to $\Pi_{v,\eps_v}$ and satisfies 
	\begin{equation} 
		\sigma_{i}(q(x)) \ge 0,  \qquad  \ m-d+2 \le  i \le m. \label{p-repr sigma ineq}
	\end{equation}
	
	Then  
	\begin{equation*} 
		0   \stackrel{\eqref{p-repr sigma ineq}}{\le} \sigma_{m-d+2}(q(x)) \nonumber \\  \stackrel{\eqref{07.03.02,13:30}, \eqref{c2 ineq}, \eqref{c1c2 ineq}}{\le} g_1(x) \left( \frac{9}{4} \sigma_1(q_v(x))^2 - |q_v(x)|^2 \right),
	\end{equation*}
	which implies that the inequality 
	\begin{equation*} -\sigma_1(q_v(x)) \ge \frac{2}{3} |q_v(x)|
	\end{equation*}
	 or the inequality 
	\begin{equation}
		\sigma_1(q_v(x)) \ge \frac{2}{3} |q_v(x)| \label{07.08.13,11:15}
	\end{equation}
	 is fulfilled. In the former case we get $x \in C_v.$ In the latter case we have
	\begin{equation*} 
		\sigma_{m-d+1}(q(x))  \stackrel{\eqref{07.03.02,15:51}, \eqref{eps:v assump}, \eqref{07.08.13,11:15}}{\ge} \frac{2}{3} |q_v(x)| \sigma_{m-d}(\bar{q}_v(x)) + r_1(x) 
			 \stackrel{\eqref{07.07.12,16:20}}{\ge} \frac{1}{3} |q_v(x)| \sigma_{m-d}(\bar{q}_v(x)) 
			 \stackrel{\eqref{eps:v assump}}{\ge} 0.
	\end{equation*}
	In view of $\eps_v \le \eps_1$ and \eqref{eps:v assump} we get $\Pi_{v,\eps_v} \subseteq P_{\eps_1}.$  Hence, by Lemma~\ref{07.06.19,09:55}, $\sigma_i(q(x)) \ge 0$ for every $1 \le i \le m-d.$ Summarizing we get that $\sigma_i(q(x)) \ge 0$ for every $1 \le i \le m.$ But then, by Proposition~\ref{orth repr}, it follows that $q_F(x) \ge 0$ for  all $F \in \calF_{d-1},$ i.e., $x \in P.$ Thus, the assertion is valid by putting $\eps_2 := \min_{v \in \vertx(P)} \eps_v.$	
\end{proof}

\begin{lemma} \label{eps3:lem}
  Let $P$ be a simple $d$-polytope. Then there exists a scalar $\eps_3>0$ such that
	\begin{equation}
		\label{07.04.11,16:30}
		\setcond{x \in P_{\eps_3}}{\SigmaQIneq}  \subseteq  P \cup \bigcup_{v \in \vertx(P)}  C_v.
	\end{equation}
\end{lemma}
\begin{proof}
	Let us choose scalars $\eps_1>0$ and $\eps_2>0$ with $\eps_2 \le \eps_1$ as in the statements of Lemmas~\ref{07.06.19,09:55} and \ref{around:vx:lem}, respectively.  If $x \in P,$ then $\sigma_{m-d+1}(q(x)) \ge 0$ with equality if and only if $x$ is a vertex of $P.$ This yields that $\sigma_{m-d+1}(q(x)) > 0$ for $x \in P \setminus \bigcup_{v \in \vertx(P)} \Pi_{v,\eps_2}.$ In view of the continuity of $\sigma_{m-d+1}(q(x)),$ there exists a scalar $\eps_3$ with $0 < \eps_3 \le \eps_2$ such that
	\begin{equation} \label{07.05.03 15:15}
		\sigma_{m-d+1}(q(x)) > 0 \qquad \mbox{for} \ x \in P_{\eps_3} \setminus \bigcup_{v \in \vertx(P)} \Pi_{v,\eps_2}.
	\end{equation}	
	Then \eqref{07.04.11,16:30} is fulfilled for $\eps_3$ as above. In fact, by construction $\eps_3 \le \eps_2 \le \eps_1.$  Let $x \in P_{\eps_3}$ be such that $\sigma_i(q(x)) \ge 0$ for $m-d+2 \le i \le m.$ If $x \in \Pi_{v,\eps_2}$ for some $v \in \vertx(P),$ by Lemma~\ref{around:vx:lem} we conclude that $x \in  C_v \cup P.$ Otherwise, $x \in   P_{\eps_3}  \setminus \bigcup_{v \in \vertx(P)} \Pi_{v,\eps_2},$ and by Lemma~\ref{07.06.19,09:55} together with \eqref{07.05.03 15:15} we deduce that $\sigma_i(q(x)) \ge 0$ for $1 \le i \le m.$ Hence, by Proposition~\ref{orth repr}, $q_F(x) \ge 0$ for $F \in \calF_{d-1},$  i.e., $x \in P.$
\end{proof}

\subsection{Approximation theorem and conclusion}

We introduce the vector $\OneVec:=[1,\ldots,1]^\transp$ from $\E^n, \ n \in \natur.$  The unit $n \times n$ matrix is denoted by $E.$ Whenever we use the notations $E$ and $\OneVec,$ the sizes of $E$ and $\OneVec$ are clear from the context. Whenever $x$ is a vector from $\E^n,$ the notation $x_i, \ i \in \setrange{1}{n},$ stands (if not endowed with another meaning) for the $i$-th component of $x.$ 
For $1 \le \nu \le +\infty$ the $l_\nu$-norm in $\E^n$  is denoted by $\lnorm{\dotvar}{\nu}.$  We also use $\lnorm{\dotvar}{\nu}$ to denote the $l_\nu$-norm of matrices induced by the vector $l_\nu$-norm. It is not hard to see that for a real matrix $A=[a_{ij}]_{i,j=1}^k$ one has 
\begin{equation} \label{mx:inf:nrm}
	|A|_\infty = \max_{1 \le i \le n} \sum_{j=1}^n |a_{ij}| \le (n-1) \max_{1 \le i,j \le n} |a_{ij}|,
\end{equation}
see, for example, \cite[Exercise~9 to Chapter~6]{MR0245579}. If $A$ is invertiable, $A^{-1}$ denotes the inverse of $A$ and $A^{-\transp}:=(A^{-1})^\transp = (A^\transp)^{-1}.$

Given compact sets $X$ and $Y$ in $\E^d$ the \notion{Hausdorff distance} between $X$ and $Y$ is defined to be the quantity 
$$
	\max \left\{ \max_{x \in X} \min_{y \in Y} |x-y| , \max_{y \in Y} \min_{x \in X} |x-y|  \right\}
$$
In what follows, the convergence of subsets of $\E^d$ will be understood with respect to the Hausdorff distance. 

Given a vertex $v$ of $P$ by $\deg_P (v)$ we denote the number of facets of $P$ incident to $v.$ We put $\deg(P):= \max_{v \in \vertx(P)} \deg_P(v).$ We also introduce a certain parameter $\gamma$ which is related to the so-called \notion{eccentricity} of a finite point set in a strictly convex position, which was introduced by Bernig, see \Bernig. We put
\begin{equation} \label{gamma:def} 
	\gamma := \max \setcond{1-q_F(v)}{F \in \calF_{d-1}, \ v \in \vertx(P) \setminus \vertx(F)}.
\end{equation}

The aim of the the following theorem is to present a construction of a convex algebraic surface which, on one hand, contains all vertices of a given polytope $P$ and, on the other hand, approximates the boundary of $P$ with any given precision. The proof of  Theorem~\ref{approx:qual}is a modification of arguments of  Bernig \Bernig[Theorem~3.1.2], who found a construction of a coinvex algebraic sufrace containg the vertices of a given $d$-polytope (without imposing however any approximation conditions). 

\begin{theorem} \label{approx:qual}
	Let $P$ a convex $d$-polytope. Then the following statements hold true. 
	\begin{enumerate}[I.]
	\item For all sufficiently large $k \in \natur$ there exist unique scalars $y_{v,k} >0,  \ v \in \vertx(P),$ such that the polynomial
	\begin{equation} \label{fk:def}
		f_k(x) := \sum_{v \in \vertx(P)} y_{v,k} \left( \frac{1}{\deg(v)} \sum_{F \in \calF_{d-1}^v} (1-q_F(x))^{2k} \right)^{2k}
	\end{equation}
	satisfies the conditions $f_k(w)=1 \ \forall \, w \in \vertx(P).$ Furthermore, the scalars $y_{k,v}, \ v \in \vertx(P),$ can be determined from the equation
	\begin{equation} \label{Ak:yk:eq}
		A_k y_k = \OneVec,
	\end{equation}
	where 
	\begin{eqnarray*}
	y_k &:=& [y_{v,k}]_{v \in \vertx(P)}, \\
	A_k &:=& [ A_k(w,v)]_{\overtwocond{w \in \vertx(P)}{v \in \vertx(P)}} := \left[ \left( \frac{1}{\deg(v)} \sum_{F \in \calF_{d-1}^v} (1-q_F(w))^{2k} \right)^{2k} \right]_{\overtwocond{w \in \vertx(P)}{v \in \vertx(P)}}.
	\end{eqnarray*}

	\item The semi-algebraic set 
	$$ S_k := \setcond{ x \in \E^d}{ f_k(x) \le 1} $$
	converges to $P,$ as $k \rightarrow +\infty.$
	\item For all sufficiently large $k$ and every $v \in \vertx(P)$  the equality $S_k \cap C_v = \{v\}$ holds.
	\end{enumerate}
\end{theorem}
\begin{proof}
	\emph{I.}  For every $v, w \in \vertx(P)$ with $v \ne w$ we have

	\begin{equation} \label{07.07.18,13:43}
		A_k(w,v)^{1/2k} = \frac{1}{\deg(v)} \sum_{F \in \calF_{d-1}^v} (1-q_F(w))^{2k} \le \frac{\deg(v) - 1 + \gamma^{2k}}{\deg(v)} \le 1 - \frac{1- \gamma^{2k}}{\deg(P)}
	\end{equation}

	and hence 
	\begin{equation}
		|A_k-E|_\infty \stackrel{\eqref{mx:inf:nrm}}{\le}  (n-1) \max_{\overtwocond{w,v \in \vertx(P)}{w \ne v}} A_k(w,v)  \stackrel{\eqref{07.07.18,13:43}}{\le} (n-1) \left( 1 - \frac{1- \gamma^{2k}}{\deg(P)}\right)^{2k} \label{dist:Ak:E:bd1}.
	\end{equation}
	The conditions $f_k(w) =1$ for $w \in \vertx(P)$ are equivalent to the system \eqref{Ak:yk:eq}. By \eqref{dist:Ak:E:bd1},  $|A_k-E|_\infty \rightarrow 0,$ as   $k \rightarrow + \infty,$ which shows that $A_k$ is invertible for all sufficiently large $k,$ and, by \eqref{Ak:yk:eq}, for every $v \in \vertx(P)$ we have $y_{v,k} \rightarrow 1,$ as $k \rightarrow +\infty.$ This shows the assertion of Part~I.

	\emph{II.} First we notice that $P \subseteq S_k,$ because $w \in S_k$ for every $w \in \vertx(P)$ and, since $f_k(x)$ is concave, $S_k$ is convex. 	If $x \in S_k,$ then  
	\begin{align*}
		1 \ge f_k(x)^{1/4k^2} & \ge  \left( \min_{v \in \vertx(P)} y_{v,k}^{1/4k^2} \right) \left( \sum_{v \in \vertx(P)} \left( \frac{1}{\deg(v)} \sum_{F \in \calF_{d-1}^v} (1-q_F(x))^{2k} \right)^{2k} \right)^{1/4k^2}  \\
	& \ge \left( \min_{v \in \vertx(P)} y_{v,k}^{1/4k^2} \right) \deg (P)^{-1/2k}  \left( \sum_{v \in \vertx(P)} \left( \sum_{F \in \calF_{d-1}^v} (1-q_F(x))^{2k} \right)^{2k} \right)^{1/4k^2}  \\
	& \ge \left( \min_{v \in \vertx(P)} y_{v,k}^{1/4k^2} \right) \deg (P)^{-1/2k} \max_{F \in \calF_{d-1}} |1-q_F(x)|, 
	\end{align*}
	and hence
	\begin{equation} \label{07.07.11,16:42} 
	S_k \subseteq \setcond{x \in \E^d}{ |1-q_F(x)| \le \frac{ \deg (P)^{1/2k}}{\min_{v \in \vertx (P)} y_{v,k}^{1/4k^2}} \ \forall \, F \in \calF_{d-1}}.
	\end{equation} 
	But since $$\frac{ \deg (P)^{1/2k}}{\min_{v \in \vertx (P)} y_{v,k}^{1/4k^2}} \longrightarrow 1,$$ as $k \rightarrow + \infty,$ and $P \subseteq S_k,$ we arrive at the assertion of Part~ II.

	\emph{III.}  We assume that $k$ is big enough so that the assertion of Part~I is fulfilled, in particular, $y_{v,k} > 0$ for every $v \in \vertx(P).$ We have 
	\begin{equation*}
		\frac{1}{ 4 k^2 } \nabla f_k(x)= \sum_{v \in \vertx (P)} y_{v,k} \left( \frac{1}{\deg(v)} \sum_{F \in \calF_{d-1}^v} (1-q_F(x))^{2k} \right)^{2k-1} \left( \frac{1}{\deg(v)} \sum_{F \in \calF_{d-1}^v} (1-q_F(x))^{2k-1} \frac{u_F}{\diam(P)} \right),
	\end{equation*}
	and thus, for $w \in \vertx(P)$
	\begin{equation*}
		\frac{1}{ 4 k^2 } \nabla f_k(w) = \frac{y_{w,k}}{\deg(w) \cdot \diam(P)} \sum_{F \in \calF_{d-1}^w} u_F  + u^w_k,
	\end{equation*}
	where 
	\begin{equation} \label{u:k:w:def}
		u^w_k := \sum_{v \in \vertx (P) \setminus \{w\}} y_{v,k} \cdot A_k(w,v)^{\frac{2k-1}{2k}} \cdot \left( \frac{1}{\deg(v)} \sum_{F \in \calF_{d-1}^v} (1-q_F(w))^{2k-1} \frac{u_F}{\diam(P)} \right).
	\end{equation}
Assume that $x \in C_w,$ that is, $-\sigma_1(q_w(x)) \ge \frac{2}{3} |q_w(x)|.$ Then
	\begin{align}
		\sprod{\frac{1}{ 4 k^2 } \nabla f_k(w)}{x-w} & =  - \frac{y_{w,k}}{\deg(w)} \cdot \sigma_1(q_w(x)) + \sprod{u^w_k}{x-w}  \ge  \frac{2}{3 \deg(w)} y_{w,k} |q_w(x)| + \sprod{u^w_k}{x-w} \nonumber \\ & \ge \frac{2}{3 \deg (P)} y_{w,k} |q_w(x)| + \sprod{u^w_k}{x-w} \label{07.07.18,14:20}
	\end{align}
	From \eqref{u:k:w:def} we see that $\sprod{u^w_k}{x-w} \le \beta(k) |q_w(x)|$ with some $\beta(k)$ converging to $0$ as $k \rightarrow +\infty.$ Thus, in view of \eqref{07.07.18,14:20}, if $k$ is sufficiently large,  we get 
	\begin{equation} \label{07.07.18,14:22}
		\sprod{\frac{1}{ 4 k^2 } \nabla f_k(w)}{x-w} \ge \frac{1}{3 \deg(P)} |q_w(x)|
	\end{equation}
	for every $x \in \E^d.$ 
   Therefore $f_k(w)$ does not vanish, and by this, is an outward normal of $S_k$ at $w,$ and moreover all points of $C_w$ distinct from $w$ lie outside $S_k.$
\end{proof}

Theorem~\ref{approx:qual}(and also its improved version  Theorem~\ref{CAAIP} given below) deals with approximation and interpolation of a convex polytope by convex semi-algebraic sets, which is also a topic of independent interest. Related results can be found in \cite{MR0154184}, \cite{MR0353146}, and \GroetschelHenk[Lemma~2.6]).


We finish the section with the proof of our main theorem.

\begin{proof}[Proof of Theorem~\ref{p-rep:simp}]
	Let $\eps_3$ be as in the assertion of Lemma~\ref{eps3:lem}. We can construct a strictly concave polynomial $p_0(x):=1-f_k(x)$ with $f_k(x)$ as in Theorem~\ref{approx:qual}and sufficiently large $k \in \natur$  such that $p_0(x)$ is non-negative on $P,$ negative on $C_v \setminus \{v\}$ for each $v \in \vertx(P)$ and $\setcond{x \in \E^d}{p_0(x) \ge 0} \subseteq P_{\eps_3}.$ Clearly, the assertion of the theorem is fulfilled for this choice of $p_0(x).$ 

	Let us describe a ``brute-force''  approach for finding an appropriate $p_0(x).$ We may assume that our input consists of polynomials $q_F(x)$ with $F \in \calF_{d-1}.$ We proceed as follows.
	\begin{enumerate}
		\item Set $k \leftarrow 1.$ 
		\item \label{iter:start} Determine the matrix $A_k$ given as in the statement of Theorem~\ref{approx:qual}.
		\item If $A_k$ is invertible, determine $y_{v,k}$ from \eqref{Ak:yk:eq}. Otherwise go to Step~\ref{iter:end}.
		\item If all $y_{v,k}$ are positive, set $p_0(x) \leftarrow 1- f_k(x)$ with $f_k(x)$ as in \eqref{fk:def}. Otherwise go to Step~\ref{iter:end}.
		\item \label{complex:step} If the polynomials $p_0(x),\ldots, p_{d-1}(x)$ represent $P,$ return $p_0(x)$ and stop.
		\item \label{iter:end} Set $k \leftarrow k+1$ and go to Step~\ref{iter:start}. 
	\end{enumerate}
	Notice that Step~\ref{complex:step} can be implemented. This is a consequence of algorithmic results on the quantifier elimination theorem, see \cite[Chapters~1, 12, and 14]{MR2248869}. 
\end{proof}

\section{Choice of parameters} \label{sec:proof:quant} 

Apparently the algorithm for determination of $p_0(x)$ described in the proof of Theorem~\ref{p-rep:simp} is highly complex. It surves a theoretical purpose of providing a relatively short confirmation of constructibility statement from Theorem~\ref{p-rep:simp}. In this section we wish to determine $k$ in a more straightforward manner by giving estimates for parameters involved in Lemmas~\ref{07.06.19,09:55}, \ref{around:vx:lem}, \ref{eps3:lem} and Theorem~\ref{approx:qual}. From the results of this section it is also clear which metric characteristic of $P$ influence $k.$

\subsection{Preliminaries}

We refer to \IneqBk for standard inequalities. It is known that for every $x \in \E^n$ one has
\begin{eqnarray}
\lnorm{x}{\nu_2} & \le & \lnorm{x}{\nu_1}, \label{l:inf:1:bounds} \\
 n^{-1/\nu_1} \lnorm{x}{\nu_1}   & \le &   n^{-1/\nu_2} \lnorm{x}{\nu_2}, \label{p:means:ineq}
\end{eqnarray}
where $1 \le \nu_1 \le \nu_2 \le +\infty.$ Formula \eqref{p:means:ineq} is the inequality for \notion{power means.} \notion{H\"older's inequality} states that 
\begin{equation} \label{Hoeld-ineq}
	|\sprod{x}{y}| \le \lnorm{x}{\mu} \lnorm{x}{\nu},
\end{equation}
for every $x, y \in \E^n$ and $1 \le \mu, \nu \le +\infty$ with $\frac{1}{\mu} + \frac{1}{\nu} = 1.$ The special case $\mu=\nu=2$ yields the \notion{Cauchy-Schwarz inequality.}

For any non-empty subset $\calX$ of $\calF_{d-1}$ we put $U_\calX$ to be the matrix with $\card \calX$ rows $u_F^\transp,$ where $F \in \calX.$ We put 
\begin{eqnarray*}
	\alpha(v) &:=& \max \setcond{ |U_\calX^{-1}|_2}{\calX \subseteq \calF_{d-1}^v, \ \card \calX = d}, \\
	\alpha &:=& \max_{v \in \vertx(P)} \alpha(P,v).
\end{eqnarray*}
The quantity $\alpha(v)$ can be viewed as anisotropy of the vertex $v$ of $P.$ If $P$ is simple we put $U_v:=U_\calX$ where $\calX:=\calF_{d-1}^v.$ In the case of simple polytopes we have
$$\alpha:=\max_{v \in \vertx(P)} |U_v^{-1}|_2.$$

We wish to bound $\alpha$ by some further metric characteristics associated with $P. $
\begin{proposition} \label{07.07.04,17:24} 
	Let $P$ be a simple $d$-polytope and let $\phi$ stands for the minimum angle between $\aff I$ and $\aff F,$ where $F$ and $I$ range over all facets and edges of $P,$ respectively, such that $I$ and $F$ have precisely one vertex in common. Then 
	\begin{eqnarray}
		\alpha &\le& \frac{\sqrt{d}}{\sin \phi}, \label{alpha:phi:bnd} \\
		\alpha & \le& \frac{\sqrt{d}}{1-\gamma}. \label{alpha:gamma:bnd} 
	\end{eqnarray}
\end{proposition}
\begin{proof}
	We use the \notion{Frobenius norm} of a matrix $A:=[a_{ij}]_{i,j=1}^n,$ which is defined by 
	\begin{equation*}
		|A|_{\mathrm{Fr}} := \left( \sum_{i,j=1}^n |a_{ij}|^2 \right)^{1/2}.
	\end{equation*}
	The norms $|A|_2$ and $|A|_{\mathrm{Fr}}$ are known to be related by 
	\begin{equation} \label{Fr:norm:ineq}
		|A|_2 \le |A|_{\mathrm{Fr}}, 
	\end{equation}
	see \cite[p.~50]{MR1927396}.

	Let us introduce vectors $a_{v,F}, \ F \in \calF_{d-1}^v,$ which are columns of $U_v^{-1},$ i.e., for $F, G \in \calF_{d-1}^v$ the quantity $\sprod{a_{v,F}}{u_G}$ is 1 if $F=G$ and 0 otherwise. Let us fix $v \in \vertx(P)$ and $F \in \calF_{d-1}^v.$ Since $P$ is simple, there exists a unique $w \in \vertx(P) \setminus \vertx(F)$ such that $v$ and $w$ are ajacent vertices of $P.$ It is easily seen that $v-w$ is parallel to $a_{v,F}.$ Let $\phi_{v,F}$ denote the angle between $\aff \{v,w\}$ and $\aff F.$ Then, using the identity $\sprod{a_{v,F}}{u_F}=1,$ we see that $|a_{v,F}| = \frac{1}{\sin \phi_{v,F}} \le \frac{1}{\sin \phi}.$ Consequently, 
	\begin{equation*}
		|U_v^{-1}|_2^2 \stackrel{\eqref{Fr:norm:ineq}}{\le} |U_v|_{\mathrm{Fr}}^2 = \sum_{F \in \calF_{d-1}^v} |a_{v,F}|^2 \le \frac{d}{\sin \phi},
	\end{equation*}
	and we get the assertion.

	Let us borrow the notations from the statement and the proof of Proposition~\ref{07.07.04,17:24}. We have 
	\begin{equation*}
		1 - \gamma \le \frac{|v-w| \cdot \sin \phi_{v,F}}{\diam(P)} \le  \sin \phi_{v,F}.
	\end{equation*}
	Since $v \in \vertx(P)$ and $F \in \calF_{d-1}^v$ are chosen arbitrarily, we get 
	$ \sin \phi \ge 1-\gamma.$ The assertion follows from \eqref{alpha:phi:bnd}.
\end{proof}

By \eqref{alpha:gamma:bnd}  we showed that $\alpha$ is bounded by a multiple of $\frac{1}{1-\gamma}.$ However, we can see that for a general simple $d$-polytope $P$ the quantities $\alpha$ and $\frac{1}{1-\gamma}$ are not of the same order of magnitude, i.e., the converse statement would not be valid. In fact let $P_l, \ l \in \natur,$ be simple $d$-polytopes that converge to some polytope $P$ which is not simple. Then $\alpha(P_l)$ converges to some finite value, as $l \rightarrow + \infty,$ however $\frac{1}{1-\gamma} \rightarrow +\infty.$

\subsection{Auxiliary statements for $P_\eps$}

The \notion{normal cone} of $P$ at a boundary point $x$ of $P$ is the set 
$$N(P,x):=\setcond{u \in \E^d}{\sprod{x}{u} = h(P,u)}.$$

\begin{lemma} \label{Peps:lem}
	Let $P$ be a simple $d$-polytope and let $\eps \ge 0$ be such that
	\begin{equation} \label{07.08.14,15:32}
		\eps <  \frac{1-\gamma}{\sqrt{d} \cdot \alpha},
	\end{equation}
	For $v \in \vertx(P)$ let $v_\eps$ be the  point determined by $d$ equalities $q_F(v^\eps)=-\eps$ with $F \in \calF_{d-1}^v.$ Then 
	\begin{align}
		\vertx(P_\eps) & = \setcond{v_\eps}{v \in \vertx(P)}, & & \label{vert:P:eps:repr} \\
		N(P_\eps,v_\eps) & = N(P,v) & &  \forall \, v \in \vertx(P),   \label{N:P:eps:repr} \\
		q_F(x) & \le 2 & & \forall F \in \calF_{d-1} \ \forall \, x \in P_\eps , \label{q:F:P:eps:up:bd} \\
		q_F(v^\eps) & \ge - \eps \cdot \sqrt{d} \cdot \alpha & & \forall v \in \vertx(P) \ \forall \, F \in \calF_{d-1}, \label{q:F:P:eps:lo:bd1} \\
		q_F(v^\eps) & \ge 1- \gamma - \eps \cdot \sqrt{d} \cdot \alpha > 0 & & \forall \, v \in \vertx(P) \ \forall \, F \in \calF_{d-1} \setminus \calF_{d-1}^v. \label{q:F:P:eps:lo:bd2}
	\end{align}
	\eop
\end{lemma}
\begin{proof}
	Since 
	\begin{equation*}
	\begin{array}{rcccccl}
		q_F(v^\eps) & = & (h(P,u_F)-\sprod{u_F}{v^\eps})/\diam(P) & = & -\eps & & \forall \, F \in \calF_{d-1}^v, \\
		q_F(v) & = & (h(P,u_F)-\sprod{u_F}{v})/\diam(P) & = & 0 & & \forall \, F \in \calF_{d-1}^v,
	\end{array}
	\end{equation*}
	we obtain
	\begin{equation*}
		\sprod{u_F}{v^\eps-v} = \eps \cdot \diam(P) \qquad \forall \, F \in \calF_{d-1}^v.
	\end{equation*}
	Then 
	\begin{equation} \label{07.07.03,15:35a}
		v^\eps -v = \eps \cdot \diam(P) \cdot U_v^{-1} \OneVec.
	\end{equation} 
	For every $v \in \vertx(P)$ and $F \in \calF_{d-1}$ we have 
	\begin{equation} 
		q_F(v^\eps)  = q_F(v) + \frac{\sprod{u_F}{v-v^\eps}}{\diam(P)} \nonumber \stackrel{\eqref{07.07.03,15:35a}}{=} q_F(v) - \eps \cdot \sprod{u_F}{U_v^{-1} \OneVec} \stackrel{\eqref{Hoeld-ineq}}{\ge} q_F(v) - \eps \cdot \sqrt{d} \cdot \alpha,
	\end{equation}
	which implies \eqref{q:F:P:eps:lo:bd1} and \eqref{q:F:P:eps:lo:bd2}. From \eqref{q:F:P:eps:lo:bd2} we deduce \eqref{N:P:eps:repr} and $\setcond{v_\eps}{v \in \vertx(P)} \subseteq \vertx(P_\eps).$ But since the cones $N(P,v)$ with $v \in \vertx(P)$ cover $\E^d,$ we obtain that the cones $N(P,v_\eps), \ v \in \vertx(P),$ also cover $\E^d$ and arrive at  \eqref{N:P:eps:repr} and \eqref{vert:P:eps:repr}.

	It remains to show \eqref{q:F:P:eps:up:bd}. Let $F \in \calF_{d-1}.$ We choose $v \in \vertx(P) \setminus \vertx(F)$ such that $h(P,-u_F)=\sprod{v}{-u_F}.$ Then, in view of \eqref{vert:P:eps:repr} and \eqref{N:P:eps:repr}, $h(P_\eps,-u_F)=\sprod{v_\eps}{-u_F}$ and we get
	\begin{align*}
		q_F(x) = \frac{h(P,u_F)-\sprod{u_F}{x}}{\diam(P)} & \le \frac{h(P,u_F)-\sprod{u_F}{v_\eps}}{\diam(P)} \le \frac{h(P,u_F)-\sprod{u_F}{v}}{\diam(P)} + \frac{\sprod{u_F}{v-v_\eps}}{\diam(P)} \\ & \le q_F(v) + \frac{|v-v_\eps|}{\diam(P)} \stackrel{\eqref{07.07.03,15:35a}}{\le} 1+ \eps \cdot \sqrt{d} \cdot \alpha \stackrel{\eqref{07.08.14,15:32}}{\le} 2 - \gamma < 2.
	\end{align*}
	arriving at \eqref{q:F:P:eps:up:bd}. 
\end{proof}

Given $v \in \vertx(P)$ and $\eps, \delta > 0$ we introduce the set
\begin{equation*}
	P^v_{\eps,\delta} := \setcond{x \in \E^d}{q_F(x) \ge -\eps \mfor F \in \calF_{d-1}^v \mand q_F(x) \ge \delta \mfor F \in \calF_{d-1} \setminus \calF_{d-1}^v},
\end{equation*}
see Fig.~\ref{PvepsdeltaFig}.  The polytope $P^v_{\eps,\delta}$ does not contain the vertex $v$ of $P$ and converges to $P$ as $\eps, \delta \rightarrow 0.$  

	\begin{FigTab}{cc}{0.6mm}
	&
	\begin{picture}(42,36)
	\put(0,0){\IncludeGraph{width=40\unitlength}{Pvepsdelta.eps}}
	\put(20,20){$P$}
	\put(0,6){\small $v$}
	\put(35,5){$P^v_{\eps,\delta}$}
	\end{picture}
	\\
	&
	\parbox[t]{0.40\textwidth}{\mycaption{\label{PvepsdeltaFig}}}
	\end{FigTab}

In the following lemma we use \notion{Carath\'eodory's theorem}. For the special case of convex polytopes it states that every point of a $d$-polytope $P$ can be represented as a convex combination of at most $d+1$ vertices of $P,$ see for example \SchnBk[Theorem~1.1.4].


\begin{lemma} \label{cover:lem} 
	Let $P$ be a simple $d$-polytope and let $\eps >0$ be such that 
	\begin{equation} \label{07.08.14,15:42}
	\delta:= \frac{1-\gamma}{1+d} - \eps \cdot \sqrt{d} \cdot \alpha > 0,
	\end{equation}
	Then 
\begin{equation} \label{P:eps:repr}
P_{\eps} = \bigcup_{v \in \vertx(P)} P^v_{\eps,\delta}.
\end{equation}
\end{lemma}
\begin{proof}
The inclusion ``$\supseteq$'' is trivial. Let us show the reverse inclusion. Since \eqref{07.08.14,15:42} implies \eqref{07.08.14,15:32} we can use Lemma~\ref{Peps:lem}.  Let the points $v_\eps$ wtih $v \in \vertx(P)$ be  defined as in the assertion of Lemma~~\ref{Peps:lem}. We fix an arbirary $x \in P_{\eps}.$ By \eqref{vert:P:eps:repr} and Carath\'eodory's theorem, there exist affinely independent vertices $v_1,\ldots, v_{d+1}$ of $P$ and non-negative scalars $\lambda_1, \ldots, \lambda_{d+1}$ such that 
	\begin{eqnarray*}
		x &=& \sum_{j=1}^{d+1} \lambda_j v_j^{\eps}, \\
		1 & = & \sum_{j=1}^{d+1} \lambda_j.
	\end{eqnarray*}
	Without loss of generality we may assume that $\lambda_1 \le \cdots \le \lambda_{d+1}.$ Then $\lambda_{d+1} \ge \frac{1}{d+1}.$ Let us choose an arbitrary $F \in \calF_{d-1} \setminus \calF_{d-1}^{v_{d+1}}.$ Then 
	\begin{align*}
		q_F(x) = \sum_{j=1}^{d+1} \lambda_j q_F(v_j^{\eps}) & \stackrel{\eqref{q:F:P:eps:lo:bd2}}{\ge} \sum_{j=1}^{d} \lambda_j q_F(v_j^{\eps}) + \frac{1}{d+1} \left(1-\gamma - \eps \cdot \sqrt{d} \cdot \alpha \right) \\ & \stackrel{\eqref{q:F:P:eps:lo:bd1}}{\ge} - \left( \sum_{j=1}^d \lambda_j \right) \cdot  \eps \cdot \sqrt{d} \cdot \alpha + \frac{1}{d+1} \left(1-\gamma - \eps \cdot \sqrt{d} \cdot \alpha \right) \\ & \stackrel{\phantom{\eqref{q:F:P:eps:lo:bd2}}}{=} - (1-\lambda_{d+1})\cdot \eps \cdot \sqrt{d} \cdot \alpha  + \frac{1}{d+1} \left(1-\gamma - \eps \cdot \sqrt{d} \cdot \alpha \right) \\ & \stackrel{\phantom{\eqref{q:F:P:eps:lo:bd2}}}{\ge} - \frac{d}{d+1} \cdot \eps \cdot \sqrt{d} \cdot \alpha + \frac{1}{d+1} \left(1-\gamma - \eps \cdot \sqrt{d} \cdot \alpha \right) \\
		& \stackrel{\phantom{\eqref{q:F:P:eps:lo:bd2}}}{=} \frac{1-\gamma}{1 + d} - \eps \cdot \sqrt{d} \cdot \alpha = \delta.
	\end{align*}
	Thus, $x \in P^{v_{d+1}}_{\eps,\delta}$ and the assertion is proved. 	
\end{proof}

\subsection{Choice of $\eps_1$, $\eps_2,$ $\eps_3$}

Lemmas~\ref{07.07.04,15:08}, \ref{eps2:qnt:lem}, and \ref{eps3:qnt:lem} below are quantiative improvements of Lemmas~\ref{07.06.19,09:55}, \ref{around:vx:lem}, and \ref{eps3:lem}, respectively.

\begin{lemma} \label{07.07.04,15:08}
Let $P$ be a simple $d$-polytope and $\eps_1 > 0$ be such that 
\begin{equation*}
	\delta:= \frac{1-\gamma}{1+d} - \eps_1 \cdot \sqrt{d} \cdot \alpha > 0,
\end{equation*}
 Then for every $1 \le i \le m-d$ and $x \in P^v_{\eps_1,\delta}$ 
\begin{equation} \label{sigm:m-d:bd}
	\sigma_i(q(x)) \ge \binom{m-d}{i} (\delta^i - 2^{i-1} \eps_1 ) + \binom{m}{i} 2^{i-1}\eps_1.
\end{equation}
In particular, $\sigma_i(q(x)) > 0$ when 
\begin{eqnarray}
	\eps_1 & \le & \left( \frac{(1-\gamma)}{4 \cdot (1+d)}\right)^{m-d}  , \label{07.07.16,14:54}\\
	\eps_1 & \le & \frac{1-\gamma}{2 \cdot (1+d) \cdot \sqrt{d} \cdot \alpha}. \label{07.07.16,14:55}
\end{eqnarray}
\end{lemma}
\begin{proof}
	In view of Lemma~\ref{cover:lem} it suffice to show \eqref{sigm:m-d:bd} for $x \in P_{\eps_1,\delta}^v$ for every $v \in \vertx(P).$ Let $v$ be fixed. 
	The quantity $\sigma_i(q(x))$ is the sum of the terms of the form $q_{F_1}(x) \cdots q_{F_i}(x)$ with $F_1,\ldots, F_i \in \calF_{d-1}.$ There are $\binom{m-d}{i}$ terms with all $F_1,\ldots,F_i$ belonging to $\calF_{d-1} \setminus \calF_{d-1}^v.$ Each of these terms is bounded from below by $\delta^{i}.$ The remaining $\binom{m}{i}-\binom{m-d}{i}$ terms might contain a negative entry $q_{F_l}, \ 1 \le l \le i,$ which is however bounded from below by $-\eps_1.$ Since, by \eqref{q:F:P:eps:up:bd}, positive entries are bounded from above by $2$ we deduce that each of these $\binom{m}{i}-\binom{m-d}{i}$ terms is bounded from below by $-2^{i-1}\eps_1 .$ The above remarks imply the assertion of the main. Now let us show the auxiliary part. We have
	\begin{equation*}
		\delta^i \stackrel{\eqref{07.07.16,14:55}} \ge  \left( \frac{1-\gamma}{2 \cdot (1+d)} \right)^i = \left( \frac{1-\gamma}{4 \cdot (1+d)} \right)^i 2^i \ge \left( \frac{1-\gamma}{4 \cdot (1+d)} \right)^{m-d} 2^{i-1} \ge \eps_1 2^{i-1},
	\end{equation*}
	which implies that $\sigma_i(q(x)) > 0.$
\end{proof}

\begin{lemma} \label{eps2:qnt:lem}
	Let $P$ be a simple $d$-polytope and let ${\eps_2} \ge 0$ be such that 
	/* Probably write $\eps_2 \le \eps_1$ with $\eps_1$ as in Lemma~** instead of the first two inequalities */
	\begin{eqnarray}
		{\eps_2} & \le & \left( \frac{(1-\gamma)}{4 \cdot (1+d)}\right)^{m-d}, \label{07.07.16,14:54a} \\
		{\eps_2} & \le & \frac{1-\gamma}{2 \cdot (1+d) \cdot \sqrt{d} \cdot \alpha}, \label{07.07.16,16:03}  \\
		{\eps_2} & \le & \frac{5 (1 - \gamma)^{m-d}}{18 \binom{d}{\floor{d/2}} 2^{m-d} (3^{m-d}-2^{m-d})} \label{07.07.16,16:41} 
	\end{eqnarray}
	Then for every $v \in \vertx(P)$ the inclusion
	\begin{equation*} 
		\setcond{x \in \Pi_{v,{\eps_2}} }{ \sigma_i(q(x)) \ge 0 \mfor m-d+2 \le i \le m} \subseteq C_v \cup P.
	\end{equation*}
	holds true.
\end{lemma}
\begin{proof}
	Let us consider an arbitrary $v \in \vertx(P).$ We borrow the notations $r_1(x), \ r_2(x), \ g_1(x), \ g_2(x)$ from the proof of Lemma~\ref{around:vx:lem}. Let $x \in \Pi_{v,{\eps_2}},$ that is
	\begin{equation}
		\label{07.08.16,17:25}
		|q_v(x)|_i\infty \le \eps_2.
	\end{equation}

	 We estimate $|r_1(x)|$ as follows: 
	\begin{align*}
		|r_1(x)| & \le \sum_{i=2}^{+\infty} |\sigma_i(q_v(x))| \cdot |\sigma_{m-d+1-i}(\bar{q}_v(x))|  \le \max_{2 \le i \le d} |\sigma_i(q_v(x))| \sum_{i=2}^{+\infty} |\sigma_{m-d+1-i}(\bar{q}_v(x))|  \\ & = \max_{2 \le i \le d} |\sigma_i(q_v(x))| \sum_{i=0}^{m-d-1} |\sigma_{i}(\bar{q}_v(x))| \le \max_{2 \le i \le d} \left( \binom{d}{i} |q_v(x)|_\infty^i \right) \sum_{i=0}^{m-d-1} \binom{m-d}{i} |\bar{q}_v(x)|_\infty^i \\
		& \le |q_v(x)|_\infty^2 \binom{d}{\floor{d/2}} \sum_{i=0}^{m-d-1} \binom{m-d}{i} |\bar{q}_v(x)|_\infty^i \stackrel{\eqref{q:F:P:eps:up:bd}}{\le} |q_v(x)|_\infty^2 \binom{d}{\floor{d/2}} \sum_{i=0}^{m-d-1} \binom{m-d}{i} 2^i \\ & = \binom{d}{\floor{d/2}} (3^{m-d}-2^{m-d}) |q_v(x)|_\infty^2,
	\end{align*}
	An analogous estimate for $|r_2(x)|$ is
	\begin{equation*}
		|r_2(x)|  \le \sum_{i=3}^{+\infty} |\sigma_i(q_v(x))| \cdot |\sigma_{m-d+2-i}(\bar{q}_v(x))| \le \binom{d}{\floor{d/2}} (3^{m-d}-2^{m-d}) |q_v(x)|_\infty^3. 
	\end{equation*}
	For every $F \in \calF_{d-1}^v$ we have 
	\begin{align*}
		q_{F}(x) & \stackrel{\phantom{\eqref{p:means:ineq}}}{=}  \frac{h(P,u_{F}) - \sprod{u_{F}}{x} }{\diam(P)}  \le \frac{h(P,u_{F}) - \sprod{u_{F}}{v} }{\diam(P)}  +  \frac{\sprod{u_{F}}{v-x}}{\diam(P)}  
		 = q_{F}(v) +  \frac{ (v-x)^\transp u_{F} }{\diam(P)}   \\
		 & \stackrel{\phantom{\eqref{p:means:ineq}}}{\ge} 1-\gamma - \left| \frac{ (v-x)^\transp u_{F} }{\diam(P)} \right| 
	  = 1-\gamma - \left| \frac{ (v-x)^\transp U_v^\transp U_v^{-\transp} u_{F} }{\diam(P)} \right| \\
		& \stackrel{\eqref{Hoeld-ineq}}{\ge}  1-\gamma - \left| \frac{(v-x)^\transp U_v^\transp}{\diam(P)} \right|_\infty | U_v^{-\transp} u_{F}|_1 
		  =  1-\gamma - |q_v(x)|_\infty | U_v^{-\transp} u_{F}|_1 \\
		 & \stackrel{\eqref{p:means:ineq}}{\ge} 1-\gamma -  \sqrt{d} \cdot |q_v(x)|_\infty | U_v^{-\transp} u_{F}|_2 
		 \stackrel{\eqref{07.08.16,17:25}}{\ge} 1-\gamma - \sqrt{d} \cdot {\eps_2} \cdot  \alpha \\ & \stackrel{\eqref{07.07.16,16:03}}{\ge} \frac{2d-1}{2(d+1)} (1 - \gamma) \ge \frac{1}{2} (1-\gamma)
	\end{align*}
	and so
	\begin{equation} \label{07.08.14,17:00}
		\sigma_{m-d}(\bar{q}_v(x)) \ge \left(\frac{1 - \gamma}{2}\right)^{m-d}.
	\end{equation}
	It suffices to show that under the given assumptions on ${\eps_2}$ inequalities  \eqref{eps:v assump}, \eqref{c2 ineq}, \eqref{c1c2 ineq}, \eqref{07.07.12,16:20} are fulfilled. Inequality \eqref{eps:v assump} was verified above. Inequality \eqref{c2 ineq} is verified as follows:
	\begin{equation*}
		g_1(x) \ge \frac{1}{2} \sigma_{m-d}(\bar{q}_v(x)) - \frac{|r_2(x)|}{|q_v(x)|_\infty^2} \ge \frac{1}{2} \left(\frac{1 - \gamma}{2}\right)^{m-d} - \binom{d}{\floor{d/2}} (3^{m-d}-2^{m-d}) {\eps_2} \stackrel{\eqref{07.07.16,16:41}}{>} 0.
	\end{equation*}
	Inequality \eqref{c1c2 ineq} is obviously equivalent to the inequality 
	\begin{equation*}
		18 \, \frac{r_2(x)}{|q_v(x)|^2} \le 5 \, \sigma_{m-d}(\bar{q}_v(x)) 
	\end{equation*}
	which is shown as follows:
	\begin{align*}
		18 \,  \frac{r_2(x)}{|q_v(x)|^2} & \le 18 \,  \frac{|r_2(x)|}{|q_v(x)|_\infty^2} \le 18 \,  \binom{d}{\floor{d/2}} (3^{m-d}-2^{m-d}) |q_v(x)|_\infty \le 18 \,  \binom{d}{\floor{d/2}} (3^{m-d}-2^{m-d}) {\eps_2} \\ & \stackrel{\eqref{07.07.16,16:41}}{\le} 5 \left(\frac{1 - \gamma}{2}\right)^{m-d} \le 5 \, \sigma_{m-d}(\bar{q}_v(x)).
	\end{align*}
	Finally we show \eqref{c2 ineq}: 
	\begin{align*}
		\frac{r_1(x)}{|q_v(x)|} & \le \frac{|r_1(x)|}{|q_v(x)|_\infty} \le \binom{d}{\floor{d/2}} (3^{m-d}-2^{m-d}) |q_v(x)|_\infty \le \binom{d}{\floor{d/2}} (3^{m-d}-2^{m-d}) {\eps_2} \\ & \stackrel{\eqref{07.07.16,16:41}}{\le} \frac{1}{3} \left(\frac{1 - \gamma}{2}\right)^{m-d} \stackrel{\eqref{07.08.14,17:00}}{\le} \frac{1}{3} \, \sigma_{m-d}(\bar{q}_v(x)).
	\end{align*}
\end{proof}

\begin{lemma} \label{eps3:qnt:lem}
	Let $P$ be a simple $d$-polytope and let 
	\begin{equation} \label{07.07.17,17:13}
		\eps_3:= \frac{1}{d-1+ \left( \binom{m}{d-1} - d \right) \left( \frac{2(1+d)}{1-\gamma} \right)^{m-d}} \cdot \eps_2
	\end{equation}
	with $\eps_2$ satisfying \eqref{07.07.16,14:54a}, \eqref{07.07.16,16:03}, \eqref{07.07.16,16:41}.

	 Then 
	\begin{equation*}
		\setcond{x \in P_{\eps_3}}{\SigmaQIneq}  \subseteq  \bigcup_{v \in \vertx(P)} C_v \cup P.
	\end{equation*}
\end{lemma}
\begin{proof}
	Let $x \in P_{\eps_3}$ be such that inequalities $\sigma_i(q(x)) \ge 0$ are fulfilled for $m-d+2 \le i \le m.$ By Lemma~\ref{cover:lem} there exists a $v \in \vertx(P)$ such that $x \in \Pi^v_{\eps_3,\delta},$ where
	\begin{equation*}
		\delta:= \frac{1-\gamma}{1+d} - \eps_3 \cdot \sqrt{d} \cdot \alpha \stackrel{\eqref{07.07.17,17:13}, \eqref{07.07.16,16:03}}{>} 0 .
	\end{equation*}

	If $x \in \Pi_{v,\eps_2},$ then, by Lemma~\ref{eps2:qnt:lem}, $x \in C_v \cup P.$  Otherwise $x \in P^v_{\eps_3,\delta} \setminus \Pi_{v,\eps_2}.$ Let us show that $\sigma_{m-d+1}(q(x)) \ge 0.$The magnitude $\sigma_{m-d+1}(q(x))$ is the sum of the terms of the form $q_{F_1}(x) \cdots q_{F_{m-d+1}}(x)$ with pairwise distinct $F_1,\ldots,F_{m-d+1}$ from $\calF_{d-1}.$ There are  $d$ such terms with precisely one $F_l, \ 1 \le l \le m-d+1,$ belonging to $\calF_{d-1}^v.$  The terms with the mentioned property sum up to $\sigma_1(q_v(x)) \sigma_{m-d}(\bar{q}_v(x)).$ Obviously, $\sigma_{m-d}(\bar{q}_v(x)) \ge \delta^{m-d}.$ For $\sigma_1(q_v(x))$ we have 
	\begin{equation} \label{07.07.04,11:41}
		\sigma_1(q_v(x)) = \sum_{F \in \calF_{d-1}^v} q_F(x)
	\end{equation}
	Let $F_0 \in \calF_{d-1}^d$ be such that $|q_{F_0}(x)|=|q_v(x)|_\infty.$ Then $|q_{F_0}(x)| > \eps_2,$ and in fact, since $\eps_2 > \eps_3$ and $q_F(x) > -\eps_3$ for every $F \in \calF_{d-1}^v,$ we even obtain that $q_{F_0}(x) > \eps_2.$ Consequently, $\sigma_1(q_v(x)) \ge \eps_2 - (d-1) \eps_3.$

	Now let us estimate the remaining $\binom{m}{d-1}-d$ terms $q_{F_1}(x) \cdots q_{F_{m-d+1}}(x)$ with pairwise distinct  $F_1,\ldots,F_{m-d+1}$ from $\calF_{d-1}$ such that at least two of the facets $F_1,\ldots,F_{m-d+1}$ belong to $\calF_{d-1}^v.$ If this kind of product $q_{F_1}(x) \cdots q_{F_{m-d+1}}(x)$ is negative then at least one entry $q_{F_l(x)}, \ 1 \le l \le m-d+1,$ lies between $-\eps$ and $0,$ while, by \eqref{q:F:P:eps:up:bd}, the remaining $2^{m-d}$ entries have absolute value at most $2.$

	Summarizing we obtain 
	\begin{equation*}
		\sigma_{m-d+1}(q(x)) \ge (\eps_2-(d-1)\eps_3) \, \delta^{m-d} - \left[ \binom{m}{d-1} -d \right] 2^{m-d} \, \eps_3.
	\end{equation*}
	Hence $\sigma_{m-d+1}(q(x)) \ge 0$  if
	\begin{equation*}
	\eps_3 \le \frac{\delta^{m-d}}{(d-1) \, \delta^{m-d} + \left(\binom{m}{d-1} - d\right) 2^{m-d}} \, \eps_2
	\end{equation*}
	But the latter inequality follows from \eqref{07.07.17,17:13}. Consequently $\sigma_{m-d+1}(q(x)) \ge 0.$ But  in view of Lemma~\ref{07.07.04,15:08}, we have $\sigma_i(q(x)) \ge 0$ for $1 \le i \le m-d.$ Summarizing we see that $\sigma_i(q(x)) \ge 0$ for $1 \le i \le m,$ and therefore, by Proposition~\ref{orth repr}, $q_F(x) \ge 0$ for all $F \in \calF_{d-1},$ i.e., $x \in P.$ 
\end{proof}

\subsection{Approximation theorem: quantitative version}

By $\log$ we denote the binary logarithm.

\begin{theorem} \label{CAAIP} 
	Let $P$ be a convex $d$-polytope and let $\eps > 0, \ k \in \natur,$ and $f_k(x), \ A_k, \ y_k, \ S_k$ be defined as in the statement of Theorem~\ref{approx:qual}. Then the following statements hold true
	\begin{enumerate}[I.]
	\item If $k$ satisfies
	\begin{eqnarray}
		k & \ge & \frac{1}{2 \, \log \frac{1}{\gamma}} ,  \label{k:gamma:bd}\\
		k & \ge & 2 \, \log (4n), \label{k:n:bd} 
	\end{eqnarray}
	then there exist unique positive real scalars $y_{v,k}, \ v \in \vertx(P),$ such that the polynomial $f_k(x)$ satisfies the condition $f_k(w)=1$ for every $w \in \vertx(P).$
	\item If $k$ satisfies \eqref{k:gamma:bd}, \eqref{k:n:bd} and 
	\begin{equation} 
		k \ge \frac{\log (2 \, \deg(P))}{2 \, \log (1+\eps)}. \label{k:eps:deg:bd} 
	\end{equation}
	and $y_k$ is determined from \eqref{Ak:yk:eq}, then the semi-algebraic set $S_k$
	satisfies the inclusions
	\begin{equation}
		P \subseteq S_k \subseteq P_\eps.
	\end{equation}
	\item If $P$ is simple and inequalities \eqref{k:gamma:bd} and
	\begin{equation} \label{07.07.06,15:22}
		k \ge 3 \log \left( 12 \, n \cdot \sqrt{d} \cdot \alpha \cdot \deg(P) \right),
	\end{equation}
	are fulfilled, then $S_k \cap C_v = \{v\}.$
	\end{enumerate}
\end{theorem}
\begin{proof}
	\emph{I.} The conditions $f_k(w) =1$ for $w \in \vertx(P)$ are equivalent to the system $A_k y_k = \OneVec.$ Let us show that under the given assumptions on $k$ the matrix is invertible. 
	\begin{align}
		|A_k-E|_\infty & \stackrel{\eqref{dist:Ak:E:bd1}}{\le} (n-1) \left( 1 - \frac{1- \gamma^{2k}}{\deg(P)}\right)^{2k} \nonumber \\
		& \stackrel{\eqref{k:gamma:bd}}{\le} (n-1) \left( 1 - \frac{1}{2 \deg(P)} \right)^{2k} \le n \left(\frac{3}{4}\right)^{2k} \stackrel{\eqref{k:n:bd}}{\le} \frac{1}{4} \label{dist:Ak:E:bd}
	\end{align}
	Thus, we have showed that $|A_k - E|_\infty < \frac{1}{2}.$ It is known that if $|A_k-E|_\infty < 1,$ then $A_k$ is invertible and moreover
	\begin{equation*}
		A_k^{-1} = \sum_{l=0}^{+\infty} (E-A_k)^l,
	\end{equation*}
	see, for example, \cite[Theorem~7.1.1]{MR0245579}. Consequently, 
	\begin{align*}
		|y_k - \OneVec|_\infty = | (A_k^{-1} - E) \OneVec |_\infty \le | A_k^{-1} - E|_\infty = \left| \sum_{l=1}^{+\infty} (E-A_k)^l \right|_\infty &\le |E-A_k|_\infty \cdot \left| \sum_{l=0}^{+\infty} (E-A_k)^l \right|_\infty \nonumber \\
			&\le \frac{|E-A_k|_\infty}{1 - |E-A_k|_\infty} \stackrel{\eqref{dist:Ak:E:bd}}{\le} \frac{1}{3} 
	\end{align*}
	and hence
	\begin{equation} \label{07.07.18,11:45}
		\frac{2}{3} \le y_{v,k} \le \frac{4}{3} \qquad \forall \, v \in \vertx(P).
	\end{equation}

	\emph{II.} The inclusion $P \subseteq S_k$ was noticed in the proof of Theorem~\ref{approx:qual}. In view of  \eqref{07.07.11,16:42}, the inclusion $S_k \subseteq P_{\eps}$ is a consequence of the following estimates: 
	\begin{align*}
		\log \frac{\deg (P)^{1/2k}}{\min_{v \in \vertx(P)} y_{v,k}^{1/4k^2}} & \stackrel{\eqref{07.07.18,11:45}}{\le} \log \deg(P)^{1/2k} 2^{1/4k^2} \le \frac{1}{2k} \log \deg(P) + \frac{1}{4k^2} \le \frac{1}{2k} \bigl(1 + \log \deg(P) \bigr) \\ & = \frac{1}{2k} \log \bigl( 2\deg(P)\bigr) \stackrel{\eqref{k:eps:deg:bd}}{\le} \log (1+\eps).
	\end{align*}

	\emph{III.} It suffices to show that under the given assumptions inequality \eqref{07.07.18,14:22} is fulfilled. 
	We have 
	\begin{align*}
		& \frac{4}{9 \deg (P)} |q_w(x)| - \sprod{\frac{1}{ 4 k^2 } \nabla f_k(w)}{x-w} \stackrel{\eqref{07.07.18,11:45}}{\le} \frac{2}{3 \deg (P)} y_{w,k} |q_w(x)| - \sprod{\frac{1}{ 4 k^2 } \nabla f_k(w)}{x-w}  \\ & \stackrel{\eqref{07.07.18,14:20}}{\le}  \sprod{u^w_k}{w-x} \le |x-w| \cdot |u^w_k| \stackrel{\eqref{u:k:w:def}}{\le} \frac{4}{3} \cdot \frac{|x-w|}{\diam(P)}  \cdot \sum_{v \in \vertx (P) \setminus \{w\}}  A_k(w,v)^{\frac{2k-1}{2k}}  \\
		& \stackrel{\phantom{\eqref{dist:Ak:E:bd1}}}{\le} \frac{4}{3} \cdot \alpha \cdot \max_{\overtwocond{\calX \in \calF_{d-1}^w}{\card \calX = d}} \frac{ |U_\calX (w-x)|}{\diam(P)} \cdot \sum_{v \in \vertx (P) \setminus \{w\}}  A_k(w,v)^{\frac{2k-1}{2k}}  \\
		& \stackrel{\eqref{p:means:ineq}}{\le} \frac{4}{3} \cdot \sqrt{d} \cdot \alpha \cdot \max_{\overtwocond{\calX \in \calF_{d-1}^w}{\card \calX = d}} \frac{ |U_\calX (w-x)|_\infty}{\diam(P)}  \cdot \sum_{v \in \vertx (P) \setminus \{w\}}  A_k(w,v)^{\frac{2k-1}{2k}} \\
		& \stackrel{\phantom{\eqref{dist:Ak:E:bd1}}}{=} \frac{4}{3} \cdot \sqrt{d} \cdot \alpha \cdot |q_w(x)|_\infty \sum_{v \in \vertx (P) \setminus \{w\}}  A_k(w,v)^{\frac{2k-1}{2k}}  
		 \stackrel{\eqref{l:inf:1:bounds}}{\le}  \frac{4}{3} \cdot \sqrt{d} \cdot \alpha \cdot |q_w(x)| \sum_{v \in \vertx (P) \setminus \{w\}}  A_k(w,v)^{\frac{2k-1}{2k}} \\
		&  \stackrel{\eqref{dist:Ak:E:bd1}}{\le} \frac{4}{3} \cdot \sqrt{d} \cdot \alpha \cdot n \cdot |q_w(x)| \left( 1 - \frac{1- \gamma^{2k}}{\deg(P)}\right)^{2k-1} 
		  \stackrel{\eqref{k:gamma:bd}}{\le} \frac{4}{3} \cdot \sqrt{d} \cdot \alpha \cdot n \cdot  \left(1 - \frac{1}{2 \deg(P)}\right)^{2k-1}  \cdot |q_w(x)|  \\
		& \stackrel{\phantom{\eqref{dist:Ak:E:bd1}}}{\le} \frac{4}{3}\sqrt{d} \cdot \alpha \cdot n \cdot  \deg(P) \cdot \left( \frac{3}{4} \right)^{2k-1} \le \frac{4}{3}\sqrt{d} \cdot \alpha \cdot n \cdot  \deg(P) \cdot \left( \frac{3}{4} \right)^{k}  \stackrel{\eqref{07.07.06,15:22}}{\le}  \frac{1}{9	 \deg(P)} |q_w(x)|,
	\end{align*}
	and we are done.	
\end{proof}

\def\cprime{$'$} \def\cprime{$'$} \def\cprime{$'$} \def\cprime{$'$}
\providecommand{\bysame}{\leavevmode\hbox to3em{\hrulefill}\thinspace}
\providecommand{\MR}{\relax\ifhmode\unskip\space\fi MR }
\providecommand{\MRhref}[2]{%
  \href{http://www.ams.org/mathscinet-getitem?mr=#1}{#2}
}
\providecommand{\href}[2]{#2}

 \begin{tabular}{l}
        Gennadiy Averkov, Martin Henk \\
	Universit\"atsplatz 2, \\
        Faculty of Mathematics \\
        University of Magdeburg\\
        39106 Magdeburg\\
        Germany \\
	\\
        \emph{e-mails:} \AverkovEmail, henk@math.uni-magdeburg.de
    \end{tabular}

 \end{document}